\documentclass[12pt,A4]{aims}
\usepackage{graphicx}
\usepackage{amssymb}
\usepackage{epstopdf}
\usepackage{amsmath,color}
\usepackage{comment}
\def\OO{\mathcal O}
\def\EE{\mathcal E}
\def\ie{{\it i.e., }}
\def\eref#1{Eq.(\ref{#1})}
\let\eqref=\eref
\def\fref#1{Fig.~\ref{#1}}
\def\bra#1{\langle #1|}\def\d{{\rm d}}
\usepackage{mhequ}
\newtheorem{theorem}{Theorem}[section]

\theoremstyle{definition}

\newtheorem{remark}{Remark}

\let\epsilon=\varepsilon
\let\theta=\vartheta
\let\phi=\varphi
\def\ipt{e^{i(t\phi+\theta)}}
\def\currentvolume{X}
 \def\currentissue{X}
  \def\currentyear{200X}
   \def\currentmonth{XX}
    \def\ppages{X--XX}
     \def\DOI{10.3934/xx.xx.xx.xx}
\title[Metastable States]{Breathers as Metastable States for the Discrete NLS equation}
\author{Jean-Pierre Eckmann and C. Eugene Wayne}
\subjclass{Primary: xxxx; Secondary: xxxx.}
 \email{jean-pierre.eckmann2unige.ch}
 \email{cew@math.bu.edu}
\thanks{The first author is supported ERC, ``Bridges'', the second is supported in part by NSF grant DMS-1311553}
\thanks{$^*$ Corresponding author: C.E. Wayne}
\begin{document}
\maketitle

\centerline{\scshape Jean-Pierre Eckmann}
\medskip
{\footnotesize
 \centerline{D\'epartement de Physique Th\'eorique}
   \centerline{and Section de Math\'ematiques}
   \centerline{Universit\'e de Gen\`eve}
   \centerline{1211 Geneva 4, Switzerland}
} 
\bigskip
\centerline{\scshape C. Eugene Wayne$^*$}
\medskip
{\footnotesize
 \centerline{Department of Mathematics and Statistics}
   \centerline{Boston University}
   \centerline{Boston, MA 02215, USA}
} 

\bigskip

\centerline{(Communicated by the associate editor name)}

\begin{abstract}
We study metastable motions in weakly damped Hamiltonian
systems.  These are believed to inhibit the transport of energy through Hamiltonian, or
nearly Hamiltonian, systems with many degrees of freedom. We 
investigate this question in a very simple model in
which the breather solutions that are thought to be
responsible for the metastable states can be computed perturbatively to an arbitrary
order.  Then, using a modulation hypothesis,   we
derive estimates for the rate at which the system drifts along this
manifold of periodic orbits and verify the optimality of our estimates
numerically. 
\end{abstract}
\emph{Dedicated to Rafael de la Llave with admiration and affection on
  his $60^{th}$ birthday.}
  
\section{Introduction}\label{sec:intro}

We are interested in the transport of energy through Hamiltonian, or
nearly Hamiltonian, systems with many degrees of freedom.  For
example, this type of transport is important for understanding heat
flow in lattice systems  \cite{politi:1997}. 
In \cite{cuneo:2017}, the authors find that in a system of $N$
rotators, coupled to their nearest neighbors, there are regions of
phase space in which the transport is very slow and they give
theoretical and numerical estimates indicating that the 
slow transport is due to the metastability of periodic breathers in
the system.   We note, that breathers and the transport of energy
in Hamiltonian lattices, have also been among the topics of R. de la Llave's
research (\cite{haro:2000}, \cite{fontich:2015}).

Here, we investigate this question further in a very simple model in
which the breathers can be computed perturbatively to an arbitrary
order.  Then, using a modulation hypothesis see \eqref{eq:modulation_hyp} we
derive estimates for the rate at which the system drifts along this
manifold of periodic orbits and verify the optimality of our estimates
numerically.

We now recall in more detail the situation considered in
\cite{cuneo:2017}.  The authors start from a Hamiltonian system of $N$
rotators coupled to their nearest neighbors.  They choose initial
conditions in which almost all of the energy is the system is
concentrated on one of the rotators at the end of the system, and they
impose a very weak damping at the opposite end.  Since it is easy to
prove that the energy eventually tends to zero, it is forced to flow
from the end of the system where it is initially located to the end
where it dissipates.  The authors find that the dissipation rate
is very slow and their numerics seem to indicate that for a very long
period, the solution remains near an approximately periodic, spatially
localized,  state reminiscent of the discrete breathers often found in
Hamiltonian lattices \cite{MacKay:1994}, \cite{Flach:1998},
\cite{Aubry:2006}. 

In the present work we look again at this problem in a finite, discrete nonlinear Schr\"odinger equation:

\begin{equ}\label{e:dNLS}
-i \frac{\partial u_j}{\partial \tau} = - (\Delta u)_j + |u_j|^2 u_j\ , \ j=1,2, \dots , N\ .
\end{equ}
Here $(\Delta u)_j= u_{j-1}-2 u_j +u_{j+1}$, with obvious
modifications for $j=1$ or $N$.
We will focus primarily on the case $N=3$, for simplicity, but in principle, our
methods apply to systems with arbitrarily many degrees of freedom, and we plan
to return to the consideration of the general case in a future work.

We will choose initial conditions for this system in which essentially all of the 
energy is in mode $u_1$, and will add a weak dissipative term to the last 
mode as in \cite{cuneo:2017} by adding to \eref{e:dNLS} a term of the form
\begin{equ}
i \gamma \delta_{N,j} u_j~,
\end{equ}
\ie we add dissipation to position $N$, at the other end from the
energetic mode.
Eventually, this will lead to the energy of the system tending to zero, but we are interested in what
happens on intermediate time scales.

If our initial conditions are chosen so that $u_1(0) = L$, and all other $u_j(0) = 0$, then we expect
that at least initially, the coupling terms between the various modes will play only a small
role in the evolution and the system will be largely dominated by the equation for $u_1$:
\begin{equ}
-i \frac{\d u_1}{\d\tau} = |u_1|^2 u_1\ ,
\end{equ}
with solution $u_1(\tau ) = e^{i L^2 \tau} L$---\ie we have a very fast rotation with large amplitude.  With this
in mind, we introduce a rescaled dependent variable and rewrite the equation in a rotating coordinate
frame by setting:
\begin{equ}\label{e:wide}
u_j(\tau) = L e^{i L^2 \tau} \tilde{w}_j(\tau)\ .
\end{equ}
Then $\tilde{w}_j$ satisfies
\begin{equ}
L^2 \tilde{w}_j  -i \frac{\partial \tilde{w}_j  }{\partial \tau}
 = - (\Delta \tilde{w})_j + L^2 |\tilde{w}_j |^2 \tilde{w}_j\ .
\end{equ}
We now add dissipation by adding a term which acts on the last
variable, with $\gamma\ge0$,
\begin{equ}
L^2 \tilde{w}_j  -i \frac{\partial \tilde{w}_j  }{\partial \tau}
 = - (\Delta \tilde{w})_j + L^2 |\tilde{w}_j |^2 \tilde{w}_j +i \gamma
\delta_{N,j} \tilde{w}_j\ .
\end{equ}
Rearranging, and dividing by $L^2$ gives
\begin{equ}
- i \frac{1}{L^2} \frac{\partial {\tilde{w}}_j }{\partial \tau}
= - \frac{1}{L^2} (\Delta \tilde{w})_j - \tilde{w}_j + |\tilde{w}_j |^2 \tilde{w}_j
+ i \frac{\gamma}{L^2} 
\delta_{N,j} \tilde{w}_j
\  .
\end{equ}
Finally, we define $\epsilon = L^{-2}$, and rescale time so that $\tau = \epsilon t$.  Setting
$w(t) = \tilde{w}(\tau)$,  we arrive finally
at 
\begin{equ}\label{e:main}
-i \frac{\partial w_j}{\partial t} = -\epsilon (\Delta w)_j - w_j + |w_j|^2 w_j+ i \epsilon \gamma\delta_{N,j} w_j~.
\end{equ}

We will study \eref{e:main} for the remainder of this paper.   We will also sometimes rewrite
this equation in the equivalent real form by defining $w_j = p_j + i q_j$, which yields the system
of equations, for $j=1,\dots,N$:
\begin{equa}[e:main_real]
  \dot q_j &= -\epsilon (\Delta p)_j
  -p_j+(q_j^2+p_j^2)p_j-\delta_ {j,N}\gamma\epsilon q_N~,\\
  \dot p_j &= \hphantom{-}\epsilon (\Delta q)_j +q_j-(q_j^2+p_j^2)q_j-\delta_ {j,N}\gamma\epsilon p_N~~.\\
\end{equa}
Note that if  $\gamma=0$, this is a Hamiltonian system with:
\begin{equa}[e:hamiltonian]
 H&=\frac{\epsilon }{2} \sum_{j<N} \left((p_j-p_{j+1})^2+(q_j-q_{j+1})^2\right)\\ &~~~-
\sum_{j=1}^N \left(\frac{1}{2}(p_j^2+q_j^2)-\frac{1}{4} (p_j^2+q_j^2)^2\right)~.
\end{equa}

\section{The existence of breathers}\label{s:breathers}

Based on the rescaling and rewriting the equation in a rotating frame, we might expect that 
when $\gamma = 0$, that is, when we have no damping, \eref{e:main}
has a fixed point, close to $w_1 = 1$, and we will prove that this is the case below.  Furthermore,
this fixed point is spatially localized---\ie its components $w_j$ decay rapidly with $j$.  Thus, it
is an example of the breather solutions often studied in Hamiltonian lattices.  Breathers often come in
families and that is the case here too---we will prove that when $\gamma=0$
\eref{e:main} has a family of periodic orbits
with frequency close to zero which are also spatially localized.

\begin{theorem}\label{th:fixed} Suppose that the damping coefficient
  $\gamma$ equals $ 0$ in
\eref{e:main}.  There exists $\epsilon_0 > 0$ such that for all $0 < \epsilon < \epsilon_0$,
the equation \eref{e:main} has a unique real valued fixed point with $w_1^* = 1 + {\OO}(\epsilon)$, 
and $w_j^* = {\OO}(\epsilon^{j-1})$, for $j=2, \dots , N$.
\end{theorem}

\begin{proof}  This theorem is a special case of Theorem \ref{th:fixed2}, and we delay the proof until
we state this theorem.
\end{proof}

Finding a periodic solution of the form \eref{e:wide} (\ie a fixed point in
the rotating coordinate system) is reduced
to finding roots of the system of equations 
\begin{equs}[e:noomega]
  0&= -\epsilon (\Delta p)_j -p_j+(q_j^2+p_j^2)p_j~,\\
  0 &= \epsilon (\Delta q)_j +q_j-(q_j^2+p_j^2)q_j~.
\end{equs}
One can easily compute, for fixed $N$ (not too large), that there is a solution
with $q^*_j=0$, and $p^*_j$  given as a series in $\epsilon$.
For example,  for  $N=3$ the series takes the form:
\begin{equa}[e:sol]
 {p_1^*}    &=1-{\frac {1}{2}}{ 
\epsilon}-{\frac {5}{8}}\epsilon^{2}-{\frac {21}{16}}{{
\epsilon}}^{3}+\OO (  \epsilon^{4}  )~,
\\
{p_2^*}   &=-\epsilon-{\frac {3}{2}}\epsilon^{2}
-{\frac {35}{8}}\epsilon^{3}+\OO (  \epsilon^{4}
  )~,\\
{p_3^*}   &=\epsilon^{2}
+{\frac {5}{2}}\epsilon^{3}+\OO   (\epsilon^{4}
  )~.
\end{equa}

\begin{remark} Since \eref{e:main}  is invariant under complex rotations $w_ j \to e^{i \theta} w_j$,
we actually have a circle of fixed points (when $\gamma=0$).  However, these are the only fixed points with $|w_1| \approx 1$.
\end{remark}

\begin{remark} Since the fixed points in Theorem \ref{th:fixed} are real valued, they correspond to
a fixed point of the system of equations \eref{e:main_real} of the form $(p^*,q^*=0)$.
\end{remark}

Numerics, for example for the case of $\epsilon=0.01$ and $N=4$, shows
that for initial conditions close to these states the periodic
solutions emerge naturally, and after a transient, the scaling of the
different components indeed matches those predicted by the theory, \ie
each oscillator has energy about $\epsilon ^2$ compared to the
previous one, see \fref{fig:1}.

\begin{figure}[ht]
\begin{center}\includegraphics[scale=1,angle=0]{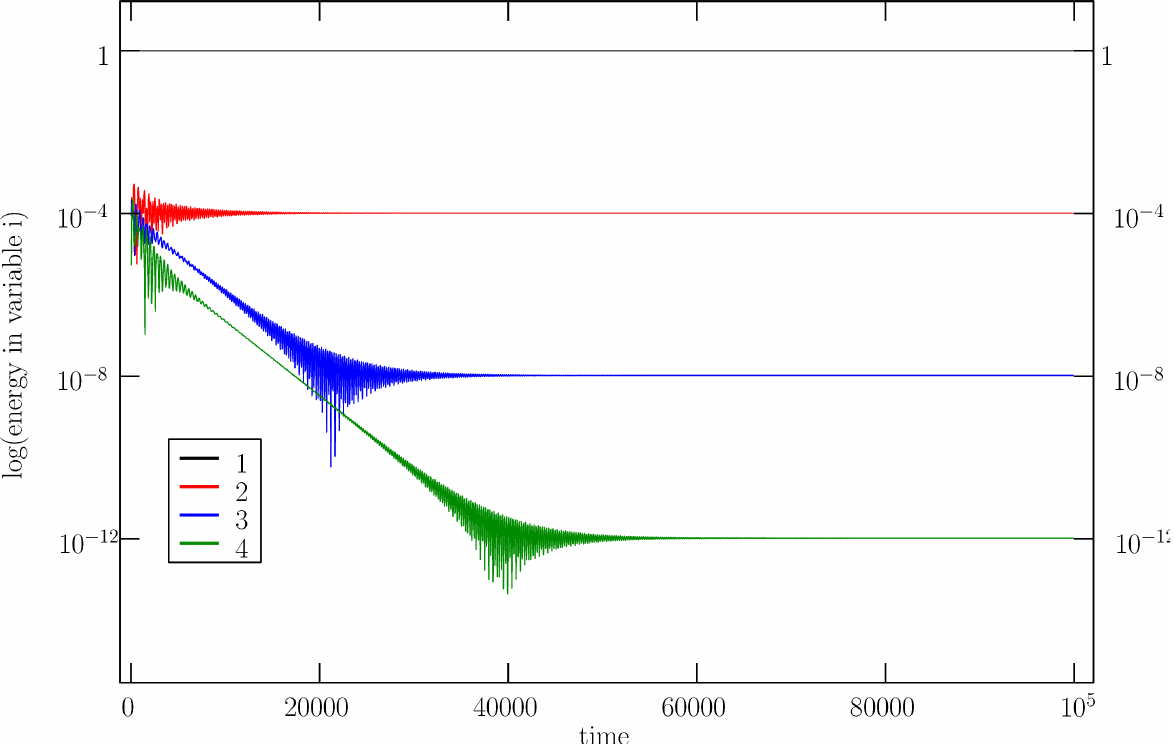}
\end{center}
\caption{Numerical illustration of the dynamics with (weak)
  dissipation. The parameters are $N=4$, $\gamma=0.2$, and $\epsilon
  =0.01$. Shown is the ``energy'' of the degree of freedom $i$,
  $i=1,\dots,4$. Note that the transients vanish after some time and
  then the energies settle at about $\epsilon ^{-2(i-1)}$. Here we
  define them as $p_i^2+q_i^2$. Note also that the dissipation is so
  slow that no decrease can be observed in the graph of $p_1^2+q_1^2$ over the time scale considered.
}\label{fig:1}

\end{figure}

We now show that there exists a family of periodic solutions near these fixed points.  We look for
solutions of the form $w_j(t;\phi) = e^{i \phi t} p_j^*(\phi)$, again with $p_j^*(\phi)$ real.  Note that
this notation is chosen so that the fixed point constructed above corresponds to the case $\phi = 0$.

\begin{theorem}\label{th:fixed2} Suppose that the damping coefficient
  $\gamma$ equals $0$ in
\eref{e:main}.  There exist constants $\epsilon_0 > 0$, $\phi_0 >0$, such that for 
$0 < \epsilon < \epsilon_0$ and $ |\phi| < \phi_0$, \eref{e:main} has a periodic solution
of the form $w_j(t;\phi) = e^{i t\phi } p_j^*(\phi)$, with $p_1^*(\phi) = 1+{\OO}(\epsilon, \phi)$,
and $p_j^*(\phi) = {\OO}(\epsilon^{j-1})$ for $j = 2, \dots, N$.
\end{theorem}

\begin{proof}

If we insert $w_j(t;\phi) = e^{i t\phi } p_j^*(\phi)$ into \eref{e:main}, and take real
and imaginary parts, we find that the amplitudes 
 $p^*$ of these periodic orbits are solutions of 
\begin{equ}\label{eq:fp}
F_j(p;\epsilon,\phi) = -\epsilon (\Delta p)_j - (1+\phi) p_j + p_j^3 = 0\ ,\ \ j = 1, \dots , N\ .
\end{equ}
Setting $p^0_j = \delta_{j,1}$, we have
\begin{equ}
F_j(p^0;0,0) = 0\ ,
\end{equ}
for all $j$.  Furthermore, the Jacobian matrix at this point is the diagonal matrix
\begin{equ}
\left( D_{p} F (p^0;0,0) \right)_{i,j} = (3 \delta_{i,1} -1) \delta_{i,j}~,
\end{equ}
which is obviously invertible.

Thus by the implicit function theorem, for $(\epsilon,\phi)$ in some neighborhood of the
origin, \eref{eq:fp} has a unique fixed point $p = p^*(\epsilon,\phi)$, and since $F$ depends
analytically on $(\epsilon,\phi)$, so does $p^*(\epsilon,\phi)$.

The decay of $p^*_j$ with $\epsilon$ follows from the analyticity of the fixed points.  Write
\begin{equ}
p^*_j(\epsilon,\phi) = p^0_j + \epsilon p^1_j(\phi) + \epsilon^2 p^2_j(\phi) \dots\ .
\end{equ}
Inserting this expansion into \eref{eq:fp} and collecting terms of order $\epsilon$ gives
\begin{equ}
(1+\phi)p^1_j - 3 (p^0_j)^2 p^1_j = -  (\Delta p^0)_j\ .
\end{equ}
From the form of $p^0$ we see:
\begin{enumerate}
\item $p^1_2(\phi) = -\frac{\epsilon}{1+\phi}$\ .
\item $p^1_j = 0\ , j \ge 3$.
\end{enumerate}
Continuing by induction we find that $p^k_{k+1} = (-1)^k \frac{1}{(1+\phi)^k}$, and
$p^k_j = 0 $ if $j \ge k+2$, from which the estimates claimed in the theorem follow.
\end{proof}

As in the case of the fixed points,  the invariance under complex rotations generates a two parameter family of periodic orbits
\begin{equ}\label{e:periodic}
w_j^*(t; \phi, \theta) = e^{i (t \phi  + \theta)} p_j^*(\phi)~.
\end{equ}

\begin{remark} Solutions similar to those considered here have been considered by many authors, including for the case of infinity many nodes.  See for example \cite{jenkinson:2016}.
\end{remark}

\section{A modulation approach to describe the motion along the periodic family}\label{s:drift}

In this section we derive equations that describe the motion of the
system near the family of periodic orbits we found above when the dissipation is included in the equations of motion.  We focus specifically
on the cases $N=2,3, {\text{ and }} 4$ sites in this section.
  We show that these orbits form a ``metastable manifold'' in the sense of \cite{beck:2010} in that once a trajectory comes close this set of states it remains near them for a very long time, moving along
the manifold of periodic solutions at a very slow rate, see \fref{fig:2}.   We calculate both the rate at which one drifts along this family, and the dissipation rate that this induces in the system and show that both of these predictions are in 
accord with our numerics.\begin{figure}[ht]
  \begin{center}
    \includegraphics[scale=0.44,angle=0]{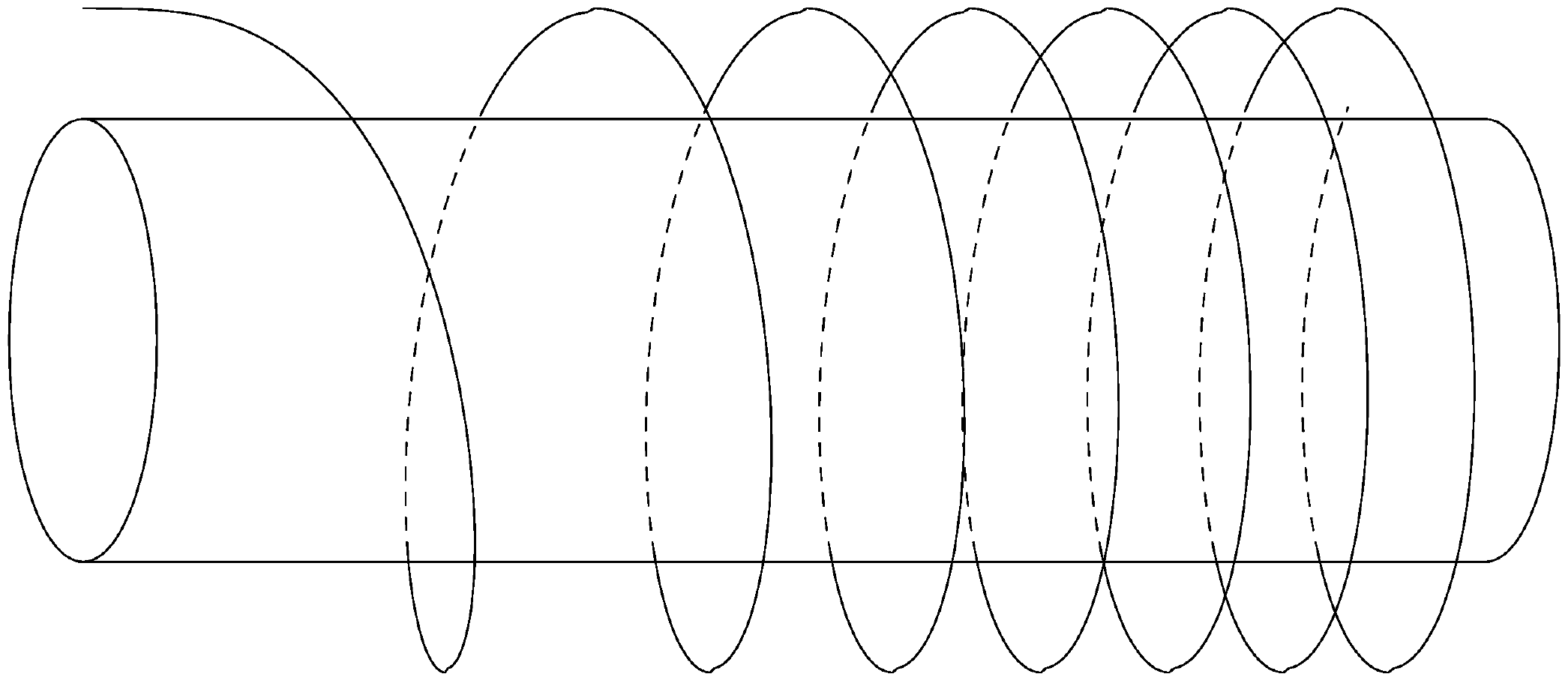}
  \end{center}
  \caption{The cylinder illustrates the set of periodic solutions of
    \eref{e:periodic}, with $\phi$ changing (very little) from left
    to right and the circle illustrating the angle $\theta$. The spiral
    illustrates the way a time-dependent solution of \eref{e:main}
    slides along the cylinder. It actually does not converge to it but
    will stay at some small, finite, distance from it. So the cylinder
    is Lyapunov stable in the sense of \cite{Arnold1989}, at
    least as long as $\epsilon $ stays small.
    }\label{fig:2}
\end{figure}

Consider the linearization of the system \eref{e:main_real} around
the periodic orbit
$w^*_j(t;\phi)  = e^{i t\phi } p_j^*(\phi)$.   
We continue with $N=3$. Setting\\ $X^*=(p_1^*,p_2^*,p_3^*,q_1^*,q_2^*,q_3^*)^{\rm T}$,
the linearization at $X^*$ takes the
form
$$
\frac{\d x}{\d t} = M x =  \begin{pmatrix}
  0&A\\B&0
  \end{pmatrix} x\ ,
  $$
with
\begin{equ}
  A=
  \begin{pmatrix}
    1-\epsilon -(p^*_1)^2 &\epsilon &0\\
    \epsilon&  1-2\epsilon-(p^*_2)^2 &\epsilon \\
    0&\epsilon&1-\epsilon -(p^*_3)^2
  \end{pmatrix}~,  
\end{equ}
where the $p^*_j$ are the solutions of \eref{e:noomega}.
Similarly,
\begin{equ}
B=
  \begin{pmatrix}
    -1+\epsilon +3(p_1^*)^2 &-\epsilon &0\\
    -\epsilon&  -1+2\epsilon +3(p_2^*)^2 &-\epsilon \\
    0&-\epsilon&-1+\epsilon +3(p_3^*)^2
  \end{pmatrix}~.
\end{equ}
Similar expressions hold for other values of $N$.

\begin{remark} By the covariance of the equations, similar formulas hold for
$\phi\sim0$ and arbitrary $\theta$, but we continue with $\phi=0$ and $\theta=0$.
\end{remark}

We next argue that the spectrum of $M$ consists of a double eigenvalue 0, and $2N-2$
imaginary eigenvalues, and we have verified this  for
$N=2,\dots,6$.

The discussion starts with $\epsilon =0$. Then, we have the ``hatted'' quantities
\begin{equ}\label{eq:Ahat}
\hat A=
\begin{pmatrix}
   0 &0 &0\\
    0&  1 &0 \\
    0&0&1
  \end{pmatrix}~,
\end{equ}
and
\begin{equ}\label{eq:Bhat}
 \hat B=
  \begin{pmatrix}
    2 &0 &0\\
    0&  -1 &0 \\
    0&0&-1
  \end{pmatrix}~.
\end{equ} 
We set 
\begin{equ}\label{eq:Mhat}
\hat M\equiv  \begin{pmatrix}
  0&\hat A\\\hat B&0
  \end{pmatrix}~,
\end{equ}
and study first the spectrum of $\hat M$. The spectrum of $M$ will
then be close to that of $\hat M$, and changing $\theta$ leaves the
spectrum unchanged.
The eigenvalues  of $\hat M$ are: A double 0, and $N-1$ pairs of eigenvalues
$\pm i$.
When $N=3$, the corresponding eigenvectors are (we write the vectors as rows):
\begin{equa}
  e^{(1)} &= (0,0,0,1,0,0)~,\\
  e^{(3),(4)} &= (0,\pm i,0,0,1,0)~,\\
  e^{(5),(6)} &= (0,0,\pm i,0,0,1)~.\\
\end{equa}
Note that $e^{(2)}$ is missing, but the vector $e^{(2)}=(1,0,0,0,0,0)$ is
mapped onto $2e^{(1)}$ and so $e^{(1)}$ and $e^{(2)}$ span the 0 eigenspace.

When $\epsilon \ne0$, the structure of eigenvectors and eigenvalues is similar to the one for $\epsilon=0$. 
Continuing with $N=3$, the null space of $M$ is spanned by
$$v^{(1)}=(0,0,0,p_1^*,p_2^*,p_3^*)~.$$ 
This comes from the invariance of \eref{e:dNLS}
under $u\to e^{i\theta}u$.  That is to say, we use the fact that $e^{i \theta} p^*_j$ is a fixed
point of the system of equations for all $\theta$, differentiate with respect to $\theta$, and set
$\theta = 0$.  Taking real and imaginary parts of the resulting equation shows that $v^{(1)}$ is
an eigenvector with eigenvalue zero.
The analog of $e^{(2)}$ is then found as follows:
From the form of $M$ we see that
\begin{equ}
  v^{(2)}=(x_1^*,x_2^*,x_3^*,0,0,0)~,
\end{equ}
with $x^*=(x_1^*,x_2^*,x_3^*)$ defined by
\begin{equ}
  x^*=B^{-1} p^*~,
\end{equ}
with $p^*=(p_1^*,p_2^*,p_3^*)$,
satisfies
\begin{equ}
  M v^{(2)} = v^{(1)}~.
\end{equ}
Note that $B^{-1}$ exists when $\epsilon $ is small enough.
Therefore, $M$ has an eigenvalue 0 and is of Jordan normal form, in
the subspace spanned by $v^{(1)}$, $v^{(2)}$. 
The vector $v^{(2)}$ is mapped by $M$ onto $v^{(1)}$, which means that $v^{(1)}$
and $v^{(2)}$ span the zero eigenspace.  As with $v^{(1)}$, $v^{(2)}$ comes
from the fact that we have a family of fixed points/period orbits, this time
depending on $\phi$.   As we showed in Theorem \ref{th:fixed2}, there is a family of fixed points of
\eqref{eq:fp} depending on $\phi$.  
Differentiating this equation with respect to $\phi$ and setting $\phi = 0$ leads
to the form of $v^{(2)}$.

The other eigenvalues are (of course) purely imaginary, 
and perturbation theory (for $N=3$) shows that they are all
simple for small enough $|\epsilon| $. If one studies the eigenvalues
of $M^2$, one finds that they are
\begin{equa}
\lambda _1&=  -1+0.76393\epsilon+0.95967477\epsilon^2+1.83724465\epsilon^3+\OO(\epsilon^4)~,\\
  \lambda _2&=-1+5.23606797\epsilon-3.95967477\epsilon^2+4.16275534829\epsilon^3+\OO(\epsilon ^4)~.
\end{equa} 
By analyticity, there is an open set of
$|\epsilon |$ where this is negative.
And the eigenvalues of $M$ are $\pm$ the square roots of
the $\lambda _k$, and so they are purely imaginary.

We next study what happens when one adds dissipation on the
coordinate $N$. In the standard basis, when $N=3$, the dissipation is
just given by $-\gamma\epsilon $ times a matrix $C$:
\begin{equ}
C=  \begin{pmatrix}
  0&0&0\\
  0&0&0\\
    0&0&1
  \end{pmatrix}~.
\end{equ}
In the full space, we have the $2N\times 2N$ matrix
\begin{equ}
  \widetilde C= 
  \begin{pmatrix}
    C&0\\0&C
  \end{pmatrix}~.
\end{equ}
Assume now that the eigenvalues are all simple. Then, for the
non-vanishing ones, one expects  that adding
dissipation of the form $\dot q_N=-\gamma\epsilon q_N$ and $\dot p_N =-\gamma\epsilon
p_N$ moves these eigenvalues into the left half plane, by an amount
proportional to $\gamma$. We illustrate this, numerically in
\fref{fig:5}, for $N=3$ with the two top lines.
\begin{figure}[ht]
\begin{center}\includegraphics[scale=1,angle=0]{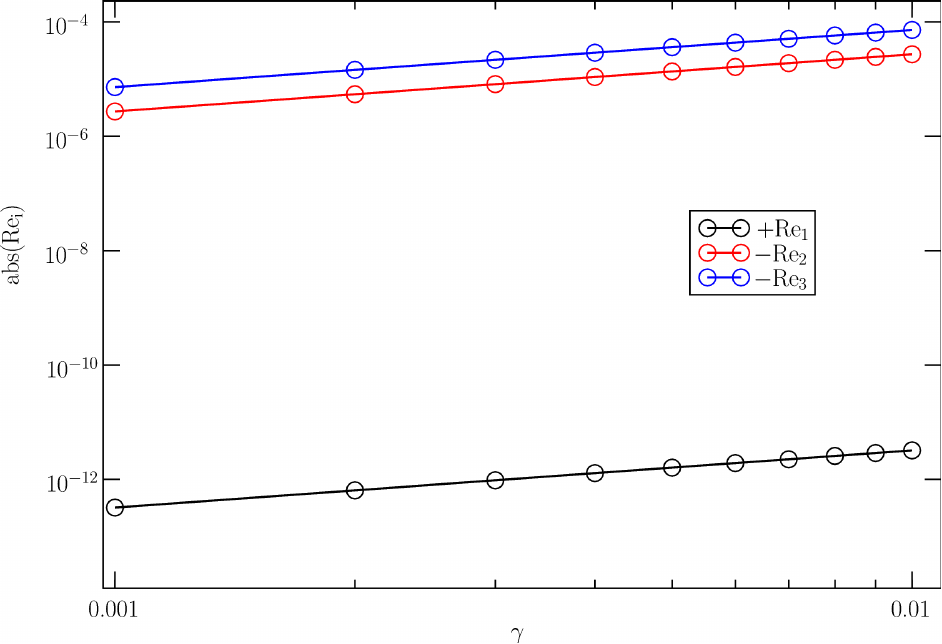}
\end{center}
\caption{The figure shows the $\gamma$ dependence of the real parts of the
eigenvalues, for $N=3$ and $\epsilon=0.01$. The three curves are linear with
intercept 0 and slopes $3.2\cdot 10^{-10}$, $0.0027$, and $0.00727$.
Note that the first eigenvalue has an extremely small \emph{positive}
real part, while the others are stable. 
}\label{fig:5}

\end{figure}
While we do not have a proof, numerical studies show clearly that the
spectrum, for $N=3,4$, when $\gamma>0$, has the following properties:
\begin{enumerate}
\item Eigenvalues which are non-zero when $\gamma =0$ acquire negative real
parts proportional to $\gamma$.  We have verified this for $N=2,3,4$ numerically.
\item The zero eigenvalues which appear when $\gamma =0$ remain very small when $\gamma$
is non-zero.  We address these eigenvalues using modulation theory, below.
\end{enumerate}

We illustrate these observations in \fref{fig:5}.   The real parts are
shown in \fref{fig:5}, where one sees that one (pair of) eigenvalues
has a \emph{positive} real part, which is very small, while the others have negative real
parts. The imaginary part of the first eigenvalue is about $1.3\cdot
10^{-5}$ when $N=3$ while the other eigenvalues have imaginary parts a
few percent less than 1 (not shown).

Our numerics indicate that as long as $\gamma$ is small, solutions remain near the 
 family of periodic orbits for very long times, even though
no true periodic orbits survive this perturbation.  We note that 
when $\gamma \ne 0$, $p^*(\phi)$ is no longer a fixed point of 
\eref{e:main_real}, so the relevance of the spectrum of $M-\gamma \epsilon \tilde{C}$ for
the dynamics of \eref{e:main_real} is not immediately apparent,
but our approach is similar to that of  \cite{promislow:2002} where as long as the ``drift''
along the family of periodic orbits is slow, we can ``freeze'' the coefficients in the
linearized vector field to understand
the dynamics over some finite interval. 

 To consider the evolution of solutions of \eref{e:main} (or equivalently, \eref{e:main_real}) for
initial conditions near the family of periodic orbits, we write the solution as
\begin{equ}\label{eq:modulation_hyp}
w_j(t) = e^{i (t\phi(t)  + \theta(t))} p^*_j(\phi(t)) + W_j(t)\ .
\end{equ}
To emphasize, if $\gamma =0$, we know that we have solutions of this form with $\phi$ and 
$\theta$ constant and $W_j(t) \equiv 0$ for all $t$. Motivated by the modulation theory approach
to understand perturbations of coherent structures (see,
{\it e.g.} \cite{pego:1992}, or \cite{promislow:2002}),
we split the various evolution equations. For this, we consider the adjoint
eigenvectors of $J\nabla H \bigr |_{x=X^*(\theta,\phi) }$,
corresponding to the  0 eigenspace.

Denote them $n^{(k)}(\theta,\phi )$ with $k=1,2$. They are defined by
\begin{equa}\label{eq:normalization}
  \langle n^{(1)},\partial _\theta X^*\rangle =\langle n^{(2)},\partial _\phi X^*\rangle&=1~,\\
  \langle n^{(1)},\partial _\phi X^*\rangle =\langle n^{(2)},\partial _\theta X^*\rangle&=0~.
\end{equa}
Due to the special and simple form of 
$
M = J \nabla H |_{x=X^*}
$,
we can write down the adjoint eigenvectors that span the zero subspace
very simply, at least when the parameter $\theta $ equals $0$.  As we found earlier, the zero eigenspace of $M$ is spanned by 
$ v^{(1)} = (0,0,0,p_1^*,p_2^*,p_3^*) = \partial_{\theta} X^*$ and 
$ v^{(2)} = (x_1^*,x_2^*,x_3^*,0,0,0) = \partial_{\phi} X^*$, and if we write out $M$ in its block form, we found $x^* = B^{-1} p^*$.

Because of the block form of $M$ and the fact that $A$ and $B$ are symmetric, we have
$$
M^{\dagger} = \left( \begin{array}{cc} 0 & B \\ A & 0 \end{array} \right)
$$
But then, since we know from the computation of the eigenvectors of $M$ that
$B p^* = 0$ (where $p^*$ is the vector corresponding to the first three components of $v^{(1)}$), we can
check immediately that
$$
\tilde{n}^{(2)} = (p_1^*,p_2^*,p_3^*,0,0,0)
$$
satisfies $M^{\dagger} \tilde{n}^{(2)} = 0$.  Likewise, 
$$
\tilde{n}^{(1)} = (0,0,0,x_1^*,x_2^*,x_3^*) = (0,B^{-1} p^*)
$$
satisfies
$$
M^{\dagger} \tilde{n}^{(1)} = (p_1^*,p_2^*,p_3^*,0,0,0) = n^{(2)}\ .
$$
Thus, $\tilde{n}^{(1)}$ and $\tilde{n}^{(2)}$ span the zero eigenspace of the adjoint operator. If we normalize these
vectors so that \eqref{eq:normalization} holds, we have $n^{(1)} = \nu_1 \tilde{n}^{(1)}$ and $n^{(2)} = \nu_2 \tilde{n}^{(2)}$
where in the case of $N=3$, the normalization constants can be shown
to satisfy $\nu_{j} = 2+\OO(\epsilon ) $, for $j = 1,2$, 
using the facts that $p_1^* = 1+ {\mathcal{O}}(\epsilon)$ (with all other components of $p^*_j$ of higher order in 
$\epsilon$) and that $B^{-1} = (\hat{B})^{-1} + {\mathcal{O}}(\epsilon)$, where $\hat{B}$ has the form \eqref{eq:Bhat}.
Thus, 
\begin{eqnarray}\label{eq:none}
n^{(1)} &=& 2 (0,0,0,x_1^*,x_2^*,x_3^*) + {\mathcal{O}}(\epsilon)\ , \\ \label{eq:ntwo}
n^{(2)} &=&  2 (p_1^*,p_2^*,p_3^*,0,0,0)+ {\mathcal{O}}(\epsilon)\ .
\end{eqnarray}

For $\theta \ne 0$,  the linearization of the vectorfield $J \nabla H |_{x=X^*}$ can be conjugated to the $\theta = 0$ case
by a rotation in the $(q_j,p_j)$ plane, so one can derive the corresponding eigenvectors by taking advantage of
this rotation.

%
%
%
%
%
%

We now assume that our initial conditions are such that we are close to one of the
periodic orbits of the undamped system, and we choose $\phi(0)$ and $\theta(0)$ so that 
$W(0)$ is orthogonal to the zero eigenspace of the linearization of the equations about this
orbit---that is, we choose  $\phi(0)$ and $\theta(0)$ so that
\begin{equ}
\langle n^{(1)}(\phi(0),\theta(0)) ~|~ W(0) \rangle = \langle n^{(2)}(\phi(0),\theta(0)) ~|~ W(0) \rangle =0\ .
\end{equ}
Inserting this form of the solution into \eref{e:main}, and expanding,
we find, writing $\phi$ for $\phi(t)$, $\theta $ for $\theta (t)$ and
$p_j^*(\phi)$ for the composition $p_j^*(\phi(t))$:
\begin{eqnarray}\nonumber\label{e:modulation} \nonumber
&&-i \dot{w}_j =\nonumber\\
&& (\phi + t \dot{\phi} + \dot{\theta})  \ipt p^*_j(\phi )
-i \dot{\phi} \ipt \partial_{\phi}  p^*_j(\phi ) - i \dot{W}_j\nonumber \\ \nonumber
&& \qquad = -\epsilon\bigl( \Delta  p^*(\phi )\bigr)_j \,\,\ipt - \ipt p^*_j(\phi )
 \\
&& \qquad \qquad  +|p^*_j(\phi)| ^2 \ipt p^*_j(\phi ) - \epsilon (\Delta W)_j
\\ \nonumber
&& \qquad -   W_j + \Bigl(2 |p^*_j(\phi )|^2 W_j + e^{2i(t\phi+\theta)} p^*_j(\phi )^2\, \overline{W}_j\Bigr)
\\ \nonumber && \qquad 
 -  i \epsilon \gamma \delta_{j,N} \ipt \Bigl(p^*_j(\phi ) + W_j(t)\Bigr )
+{\OO}(|W_j|^2)\ .
\end{eqnarray}

Recall that $p^*_j$ is a periodic orbit of the $\gamma=0$ equations, so
\begin{equ}
0 = - \epsilon(\Delta p^*)_j - (1+\phi) p^*_j + |p^*_j|^2 p^*_j\ .
\end{equ}
This leads to the cancellation of many terms in \eref{e:modulation} and we are left with
\begin{equa}[e:mod2]
(t \dot{\phi}(t) &+ \dot{\theta}) p^*_j - i \dot{\phi} (\partial_{\phi} p^*_j) - i \dot{W}_j\\ \nonumber
= &- \epsilon (\Delta W)_j -  W_j\\& +  \Bigl(2| p^*_j(\phi )|^2 W_j + e^{2i (t\phi  + \theta)} p^*_j(\phi )^2 \overline{W}_j\Bigr) 
\\  &+ i \epsilon \gamma \delta_{j,N} p_j^* + i \epsilon \gamma \delta_{j,N} W_j
+ {\OO}(|W_j|^2)\ .
\end{equa}

We have defined $W$ so that it lies (at least initially) in the rapidly 
damped subspace of the linearized equations,
 and hence we expect it to decay quickly to a small value and remain small.  We plan
to return to this point in future research (using the methods of \cite{promislow:2002}) but for the moment
we set $W \equiv 0$ and focus on how the parameters
$\phi(t),\theta(t)$ evolve---\ie how
we move along the family of periodic orbits, assuming
that we remain very close to it.  If we set $W \equiv 0$, in \eref{e:mod2} we are left
with
\begin{equ}\label{eq:modpq}
(t \dot{\phi}(t) + \dot{\theta}) p^*_j - i \dot{\phi} (\partial_{\phi} p^*_j) 
=   i \epsilon \gamma \delta_{j,N} p_j^* ~.
\end{equ}
To isolate the equations for $\dot{\theta}$ and $\dot{\phi}$ we separate the real and imaginary
parts of this equation, and rewrite it in vector form as $X = (p,q)$, where as above, $q$ represents
the imaginary part of the complex vector, and $p$ its real part.   We also recall that 
the vectors $v^{(1)}$ and $v^{(2)}$ which span the zero eigenspace of the linearization
of the differential equation are written in this imaginary/real decomposition as
$v^{(1)} = (0,0,0,p^*_1,p^*_2,p^*_3)$ and 
$v^{(2)} = ((\partial_{\phi} p^*)_1,(\partial_{\phi} p^*)_2,(\partial_{\phi} p^*)_3,0,0,0)$.  Thus, 
if we multiply \eqref{eq:modpq} by ``i'', we see that in
the case of $N=3$, \eqref{eq:modpq} is
equivalent to 
\begin{equ}
(t \dot{\phi} + \dot{\theta}) v^{(1)}  + \dot{\phi} v^{(2)} = (0,0,-\epsilon\gamma p_3^*,0,0,0)\ .
\end{equ}
Applying the projection operators $\bra{ n^{(1)}} $ and 
$\bra{ n^{(2)}}$ to this equation we obtain:
\begin{equa}
&(t \dot{\phi}(t) + \dot{\theta})  =\,0~, \\
&\dot{\phi}  = 
\,-\epsilon \gamma  p_3^* (n^{(2)})_3 \approx - 2 \epsilon \gamma  (p_3^*)^2 
\approx   - 2 \epsilon^5  \gamma  \ ,
\end{equa}
where the next to last equality used the formula for $n^{(2)}$ derived in \eqref{eq:ntwo} and 
the last equality used the form of $p^*_3$ derived in Theorem \ref{th:fixed2}.  Here, the approximate
equalities mean that we have retained only the lowest order terms in $\epsilon$ in these expressions.

From this we obtain
\begin{equa}
\phi(t) &=  -2  \gamma \epsilon^5 t~, \\
\theta(t) &=  \gamma \epsilon^5 t^2~.
\end{equa}

Note that this gives us an easy way to check the progress along the family of periodic states.
Recall that our modulation hypothesis implies that
\begin{equ}\label{e:mod_approx}
w_j(t) \approx e^{i (t\phi(t)  + \theta(t))} p_j^*(\phi(t))
= \exp({ - i  \gamma \epsilon^5 t^2) }\cdot p_j^*(\phi(t))\ .
\end{equ}
One particularly easy thing to check numerically is the point at which the first component of the 
computed solution is purely real (\ie when $q_1(t) = 0$).
If we denote the $k^{\rm th}$ time at which this occurs as $T_k$, then from
\eref{e:mod_approx} we see that
\begin{equ}\label{eq:crossing_time}
  \gamma \epsilon^5 T_k^2=  2\pi k\ .
\end{equ}

From this, we see immediately that 
\begin{equ}
\frac{T_{k+1} }{T_k} = \sqrt{ \frac{k+1}{k}}\ .
\end{equ}

\begin{figure}[ht]
  \begin{center}
    \includegraphics[width=0.7\textwidth,angle=270]{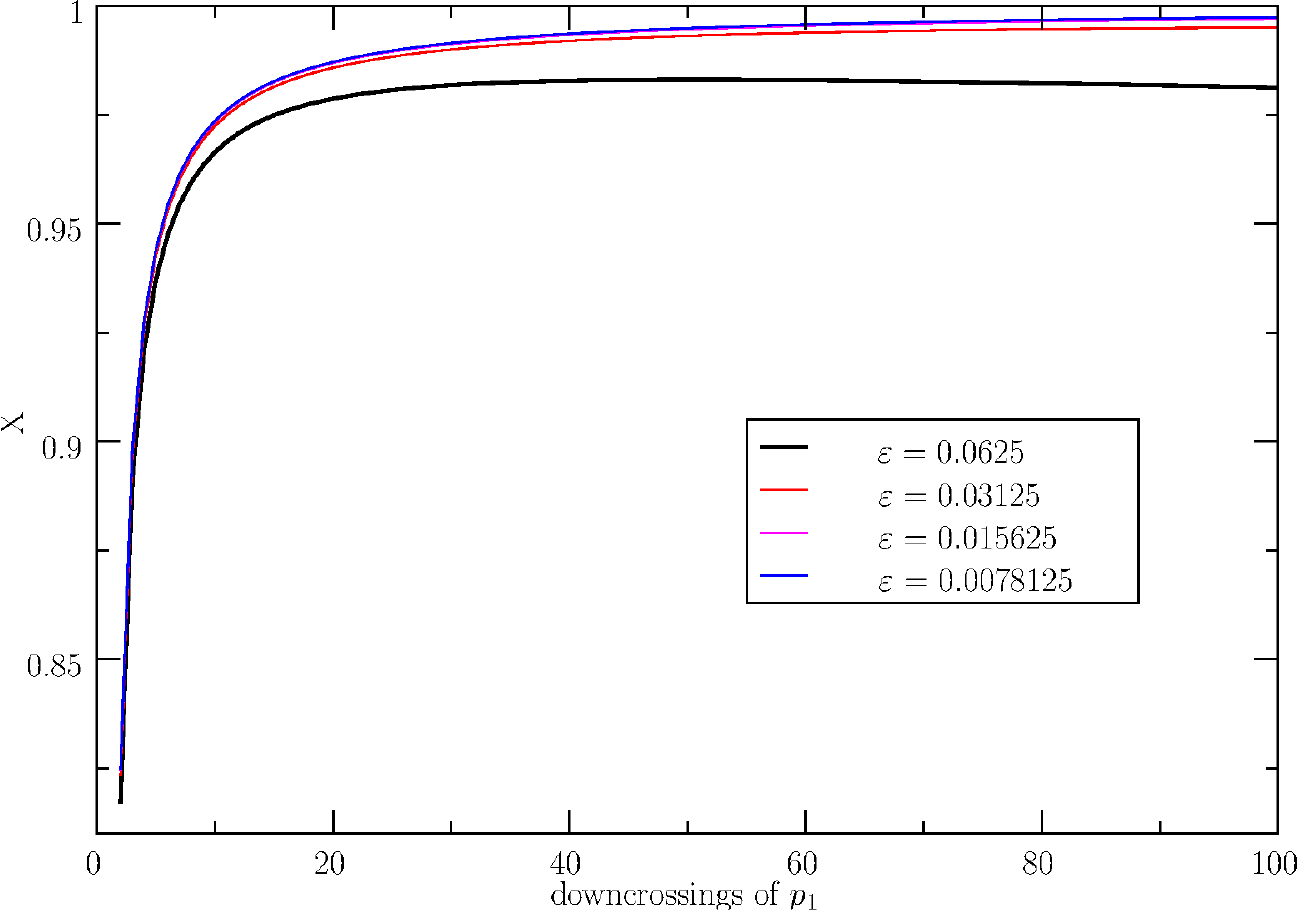}
  \end{center}
  \caption{This graph illustrates the behavior of $p_1(t)$, for $N=3$
    and several values of $\epsilon $. and $\gamma=0.2\epsilon $. One
    measures the downcrossing times $T_k$ of the $k^{\rm th}$ downcrossing
    of $p_1$ through 0.  The theory predicts
that\newline $X=\bigl({T_k/T_{k-1}}-1\bigr)/(\sqrt{k/(k-1)}-1) = 1$.  As noted in the text,
the transient behavior is not yet understood. }\label{fig:3}
  \end{figure}
Numerical computations of $T_k$ exhibited in \fref{fig:3} show clearly this
scaling with $\sqrt{k}$ indicating that our modulation hypothesis is capturing 
the drift along the family of periodic orbits.  However, \eqref{eq:crossing_time} indicates that this
rotation of the phase should begin immediately, while our numerics indicates that there is some initial
transient, and the $\sqrt{k}$ scaling appears only after some time.
We are not yet sure what the origin
of this transient behavior is, and plan to investigate it further in future work.

As a second check of the validity of this scenario we compute the decay of
energy in the system as captured by the $\ell^2$ norm of the solution.
Define
\begin{equ}
\EE(t) = \frac{1}{2} \sum_{j=1}^N (p_j(t)^2+q_j(t)^2)\ .
\end{equ}
\begin{figure}[t!]
  \begin{center}
    \includegraphics[width=0.7\textwidth,angle=270]{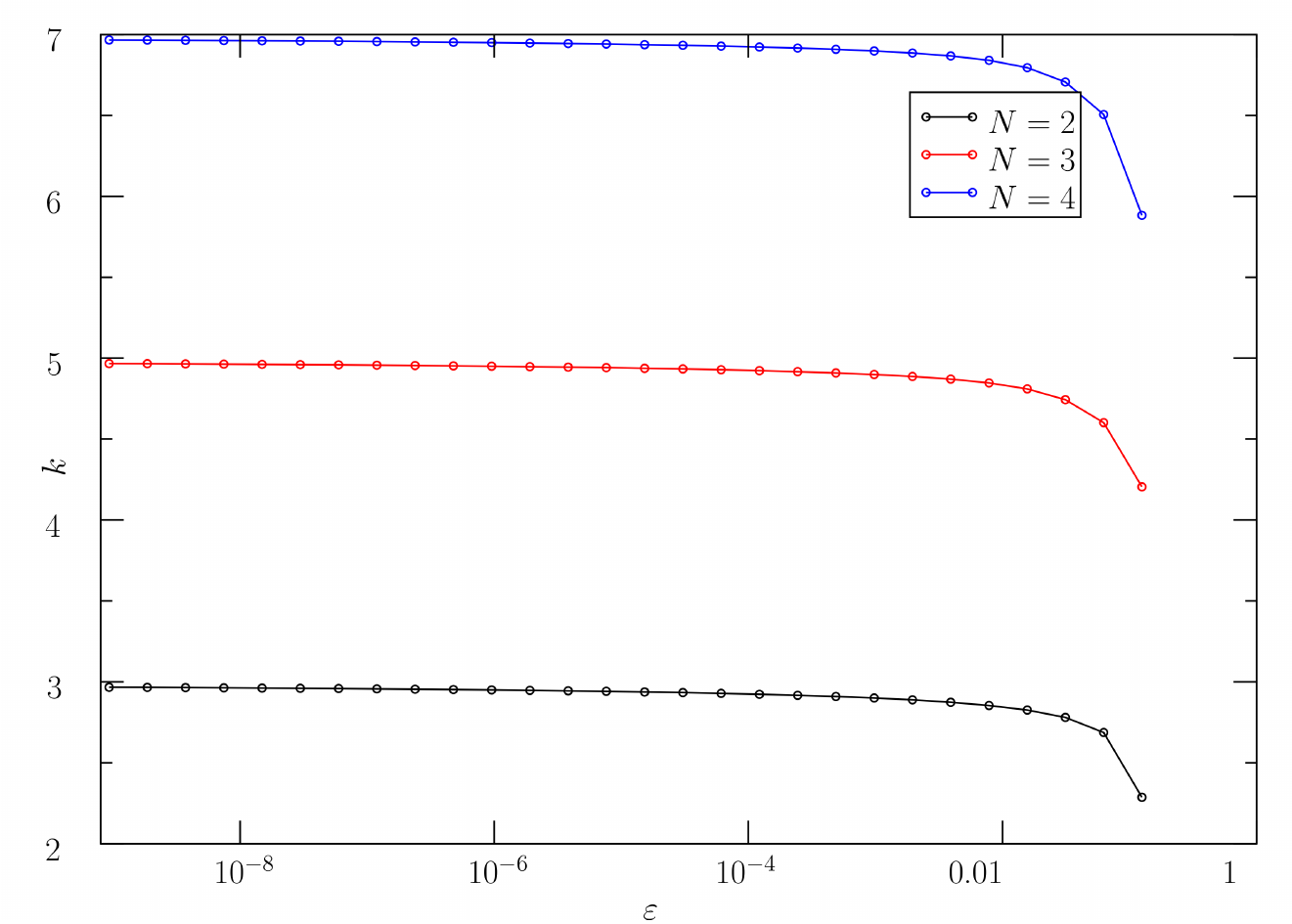}
  \end{center}
  \caption{This graph illustrates the decay properties of the $\ell_2$
    norm as a function of time, for various values of $N$ and
    $\epsilon$. 
The horizontal axis is $\epsilon $ and the vertical axis is an
    estimate of the decay rate, obtained as follows: If $m_t$ is the
    $\ell_2$ norm at time $t$ and $m_{t'}$ that at time $t'$, then we
    compute
 $   k=k(\epsilon )=\log\left( \frac{\log(m_t/m_{t'})}{\gamma\cdot
      (t'-t)}\right)/\log (\epsilon ). 
$
If $m_t$ decays like $\exp(-ct\gamma\epsilon^s)$, then the calculation will
  lead to $k=s$. Indeed, we see that the decay rate is $\epsilon
  ^{2N-1}$. The calculations shown were done for $\gamma=0.2$.
}\label{fig:4}
\end{figure}

The undamped $\gamma =0$ dynamics preserve this energy.  If we differentiate with respect
to time, the only contribution comes from the damping applied to the
$N^{\rm th}$ mode and we
find
\begin{equ}
\dot{\EE}(t) = - \gamma \epsilon (p_N(t)^2 + q_N(t)^2)~.
\end{equ}

If we again appeal to our modulation hypothesis that the motion remains close to the periodic
orbit, we can use the values of $(p_N(t),q_N(t))$ from 
\begin{equ}
w_t(t) \approx e^{i (t\phi(t)  + \theta(t))} p_j^*(\phi(t))
= \exp( -i \gamma \epsilon^5 t^2 )\cdot p_j^*(\phi(t))\ .
\end{equ}
The complex phase makes no contribution and we are left with
\begin{equ}
\dot{\EE}(t) = - \gamma \epsilon (p_N^*(\phi(t)))^2\ .
\end{equ}
In the case of $N=3$ modes, this leads to
\begin{equ}
\dot{\EE}(t) = -\gamma \epsilon^5 + \OO(\epsilon^6)\ .
\end{equ}
Once again, we check this numerically in \fref{fig:4}.

\section{Conclusions}  In a simple model, we have shown that one can compute a family of breather solutions perturbatively.
We have then shown the link of these breathers to the extremely slow rate of energy dissipation observed in such
lattices of coupled nonlinear oscillators, by using a modulation hypothesis to compute the rate at which trajectories
drift along the family of breathers in the limit of small dissipation, and we have then linked the predictions of the modulation
theory to the rate of energy dissipation.  The predictions of the theory agree very well with numerical experiments performed
on this system.
\medskip
\section*{Acknowledgments}We thank No\'e Cuneo for very helpful
discussions related to these systems.
\null

\bibliographystyle{unsrt}
\bibliography{NLS}
\medskip
Received xxxx 20xx; revised xxxx 20xx.
\medskip
\end{document}